\documentclass[0pt]{amsart}
\usepackage{times}
\usepackage{amsfonts}
\usepackage{eucal}
\usepackage{amsmath}

\newtheorem{thm}{Theorem}
\newtheorem{lemma}[thm]{Lemma}
\newtheorem{cor}[thm]{Corollary}
\newtheorem{prop}[thm]{Proposition}

\newtheorem{rem}[thm]{Remark}

\begin{document}

\title[ Individual invariance principle ]
{A Sobolev inequality and the individual invariance principle for diffusions in a periodic potential.  }

\author{Moustapha BA    \qquad   \qquad\and      \qquad  Pierre MATHIEU}



\begin{abstract}

We consider a diffusion process in $\mathbb{R}^d$ with a generator of the form $ L:=\frac 12 e^{V(x)}div(e^{-V(x)}\nabla ) $ where $V$ is measurable and periodic. 
We only assume that $e^V$ and $e^{-V}$ are locally integrable. 
We then show that, after proper rescaling, the law of the diffusion converges to a Brownian motion for Lebesgue almost all starting points.  

This pointwise invariance principle was previously known under uniform ellipticity conditions (when $V$ is bounded), see \cite{BLP} and \cite{L},  
and was recently proved under more restrictive $L^p$ conditions on $e^V$ and $e^{-V}$ in \cite{ADS}.  

Our approach uses Dirichlet form theory to define the process, martingales and time changes and the construction of a corrector. 
Our main technical tool to show the sub-linear growth of the corrector is a new weighted Sobolev type inequality for integrable potentials. 
We heavily rely on harmonic analysis technics.

\end{abstract}

\maketitle

\section{Introduction} 

We are interested here in diffusion processes on $\mathbb{R}^d$ $d\geq 2$ driven by a linear second-order divergence form operator of the type: 
\[ L:=\frac 12 e^{V(x)}div(e^{-V(x)}\nabla )\qquad\hbox{where $V:\mathbb{R}^d\rightarrow \mathbb{R}$ is measurable and periodic}.\] 

When $V$ is  assumed to be regular, the diffusion process generated by $L$ can be constructed as a solution of the stochastic differential equation:  
\begin{equation} \label{eq:sde}dX_t= dB_t-\frac 12\nabla V(X_t)dt, \end{equation} 
where $(B_t\,;\, t\ge 0)$ is a standard  Wiener process on $\mathbb{R}^d$.   The stochastic process $(X_t\,;\, t\ge 0)$ is then a semi-martingale and It\^o's stochastic calculus can be applied.
  
To make sense of equation (\ref{eq:sde}) in the more general case where $V$ is only assumed to be measurable, we shall use Dirichlet form theory. In Section \ref{sec:existence}, we assume that $e^V$ and $e^{-V}$ are both locally integrable, and show the existence of a Markovian law 
on path space $C([0,+\infty); \mathbb{R}^d)$ with generator $L$.  
The stochastic calculus developed in \cite{FUK} will play a key role. 

\medskip 

Such equations as (\ref{eq:sde}) model the motion of a passive tracer submitted to two effects: a diffusion movement represented by the Brownian motion $B$ and 
an external force described by the potential $V$. 

Many works in the domain of homogenization theory addressed the question of the long-time behavior of such diffusions.  Two cases are generally studied: either the potential is periodic or it is a realization of a stationary random function. Clearly the first can be seen as a special case of the second. Also many results hold for similar discrete models where $\mathbb{R}^d$ is replaced by the grid $\mathbb{Z}^d$ and one studies so-called {\it random walks with random conductances}.  

Homogenization theory states that, under appropriate restrictions on $V$, solutions of elliptic problems associated to the operator $L$ on, say, a large ball, scale to solutions of similar problems where $L$ is replaced by an {\it homogenized} operator with constant coefficients, say 
\[\bar{L}=\frac 12 \sum_{i,j}(\bar\sigma)_{i,j}\partial_i\partial_j,\] where $\bar{\sigma}$ is a positive symmetric matrix, the so-called {\it effective diffusivity}. 

In probabilistic terms, proving homogenization results amounts to showing the   
rescaled process $(X^{(\epsilon)}_t:=\epsilon X_{t/{\epsilon^2}}\,;\, t\geq 0)$ satisfies a {\it functional 
central limit theorem} \\ - or {\it invariance principle}. Namely one shows that the distribution of the process $X^{(\epsilon)}$, on the space of continuous functions from $[0,+\infty)$ which values in $\mathbb{R}^d$, weakly 
converges to the law of a Brownian motion with covariance matrix $\bar\sigma$. 

\medskip

Let us now describe more precisely the different results that one finds in the literature and that are relevant here. 

We let $I_0:=\mathbb{R}^d/\mathbb{Z}^d$ be the unit torus. The potential $V$ is assumed to satisfy $V(x+z)=V(x)$ for all $x\in\mathbb{R}^d$ and $z\in\mathbb{Z}^d$. 
We may sometimes identify $I_0$ with a cube in $\mathbb{R}^d$. 

We use the notation $(X_t\,;\,t\ge0)$ to denote the canonical process on $C([0,+\infty); \mathbb{R}^d)$ and 
$P_x$ to denote the law of the process generated by $L$ with starting point $x\in\mathbb{R}^d$. Also denote with 
\[ 
P_u(.):=\int_{I_0} P_x(.)\,  dx, 
\]  the law of the process when starting with uniform law on $I_0$, 
and more generally 
\[ 
P_w(.):=\int P_x(.)w(x)\,  dx, 
\]  the law of the process when the initial law has density $w$ with respect to $dx$.  

\medskip 

In \cite{BLP}, the authors assume the function $V$ is smooth. Observe it implies that $V$ is bounded. They use the stochastic differential equation (\ref{eq:sde}) to define the process $X$ for any given initial point $x\in\mathbb{R}^d$ and establish the invariance principle under $P_x$ for any $x\in\mathbb{R}^d$. 

These results were later generalized in \cite{L} to the case of a measurable and bounded potential $V$. Then the construction of the process is based on Dirichlet form theory. Observe however that when $V$ is bounded, the operator $L$ is then uniformly elliptic, so that all kind of a-priori Gaussian bounds and H\"older regularity estimates are known to hold for the fundamental solution of $L$. These in particular allow to define $P_x$ for all $x\in\mathbb{R}^d$. Another consequence is that it is then sufficient to prove the invariance principle under $P_u$. Indeed one may combine H\"older regularity estimates and the invariance principle under $P_u$ to deduce it under $P_x$ for any $x\in\mathbb{R}^d$. 

The singular case - when $V$ is not assumed to be bounded anymore - is considered in \cite{PZ} (as a special case of diffusions in a random environment). 
The authors assume that both $e^V$ and $e^{-V}$ are locally integrable and they use $2$-scale arguments to show homogenization results and the central limit theorem under $P_u$: 
the law of $X_t/\sqrt{t}$ under $P_u$ converges to the Gaussian distribution with covariance $\bar\sigma$.  

An alternative approach, which also applies to random environments, was previously developed  in \cite{DeMasi}. It is based on the interpretation of the process $X$ as an additive functional of a reversible Markovian dynamics, the so-called process of the 
{\it environment seen from the particle}. In our context, the process of the environment seen from the particle is just the projection of $X$ on the torus $I_0$. 
Applying the general results from \cite{DeMasi} in the periodic setting, one gets a functional central limit theorem under $P_u$ if $V$ is such that $\nabla V$ is integrable and $e^{V} + e^{-V}\in L^1(I_0; dx)$, see part 6 in  \cite{DeMasi}. 
It is quite possible that, at the cost of some extra work, one can remove the assumption on $\nabla V$ and then, still using the arguments in \cite{DeMasi}, obtain the invariance principle under the only assumption that $e^{V} + e^{-V}\in L^1(I_0; dx)$. Observe however that, as in \cite{PZ}, the approach in \cite{DeMasi} can only give {\it averaged}  results under $P_u$ and does not tell us anything on the behavior of the process under $P_x$ for a given starting point $x$. 

\medskip 

The question which interests us in this paper is to show the {\it individual} invariance principle without assuming $V$ is bounded. Namely we wish to show that, under $P_x$, for a given $x$, the process scales to Brownian motion. Note however that the approach through Dirichlet form only provides a definition of $P_x$ for $x$ outside a set of zero Lebesgue measure. Our main result is the following: 
\begin{thm}
\label{th} 
Assume that $e^{V} + e^{-V}\in L^1(I_0; dx)$. 
There exists a positive symmetric non-degenerate matrix $\bar{\sigma}$  such that for almost all x $\in \mathbb{R}^d$, under $P_x$, the family of processes $({X}^{(\epsilon)}\,;\,\epsilon>0)$ converges in distribution, as $\epsilon$ tends to zero,  
towards the law of a Brownian motion with covariance matrix $\bar{\sigma}$. 
\end{thm}

We note that the integrability condition $e^{V} + e^{-V}\in L^1(I_0; dx)$ is reasonable. 
On the one hand, it arises naturally when one tries to prove the existence of the process through constructing its Dirichlet form, see Part \ref{sec:existence}. 
On the other hand, in the case $d=1$, it is known that the convergence of $X^{(\epsilon)}$ towards a non-degenerate Brownian motion holds if and only if $e^{V} + e^{-V}\in L^1(I_0; dx)$, see \cite{DD}. It does not mean the condition $e^{V} + e^{-V}\in L^1(I_0; dx)$ is always necessary for the individual functional C.L.T. to hold. Indeed one might think of examples of perforated environments, where $V$ takes the value $+\infty$ on a set of non zero measure, and nevertheless the individual functional C.L.T. may hold.  

Our individual invariance principle for almost any starting point $x$ corresponds to what is known in the more general context of random environments as a {\it quenched} invariance principle where one gets a functional C.L.T. for a given starting point and almost any realization of the environment. 


In the context of random walks with random conductances, a lot of effort was recently made to get quenched invariance principles. In particular it was recently proved in \cite{ADS} that the quenched functional C.L.T. holds for random stationary conductances satisfying some moment conditions. Observe however that the moment condition used in \cite{ADS} is much more restrictive than ours. In particular it gets worse as the dimension grows. 

\medskip

Our strategy for proving Theorem \ref{th} follows some classical steps: we rely on the construction of the so-called {\it corrector}: 
this is a periodic function $v:\mathbb{R}^d\rightarrow\mathbb{R}^d$ such that the process 
$t\rightarrow X_t+v(X_t)$ is a martingale with stationary increments under $P_x$. It then follows that the process $X^{(\epsilon)}+\epsilon v(\frac 1\epsilon X^{(\epsilon)})$ satisfies the invariance principle, see part \ref{sec:corr}, 
and the key step of the proof of the Theorem consists in showing that the corrector part $\epsilon v(\frac 1\epsilon X^{(\epsilon)})$ tends to $0$. 

In order to control the corrector, and actually also in order to show its existence, we rely on the following Sobolev inequality: 

\begin{thm}
\label{th:poinc}
Let $V$ be a measurable function defined on $I_0$  satisfying\\ $e^{V}+ e^{-V}\in L^1(I_0; dx)$. 
Then there exists a positive and bounded function $w$,  there exists $r^* >2$ and there exists a constant $c$ such that: 
\begin{equation}\label{eq:sob}
 \left(\int_{I_0}\left|f(x)\right|^{r^*} w(x)\,dx\right)^{2/r^*}\leq c \int_{I_0}\left|\nabla f(x)\right|^2 e^{-V(x)} dx. 
\end{equation}
for all  function $f$  defined on $I_0$, centered and $C^1$ there. 
\end{thm}

Theorem \ref{th:poinc} is proved in Part \ref{sec:sobolev}.

Once this Sobolev-type inequality is proved, we may copy the strategy of \cite{PM}: 
we derive a first invariance principle for a time-changed version of the process $X$ and finally prove Theorem \ref{th} in Part \ref{sec:hom}. 

We believe the Sobolev inequality from Theorem \ref{th:poinc} has its own interest.

\section{Dirichlet forms and processes} 
\label{sec:existence} 

\newcommand{\dL}{\dot{L}}
\newcommand{\dx}{\dot{x}}
\newcommand{\dy}{\dot{y}}
\newcommand{\dX}{\dot{X}}
\newcommand{\dxi}{\dot{\xi}}
\newcommand{\dP}{\dot{P}}

We recall that $I_0$ stands for the unit torus: $I_0:=\mathbb{R}^d/\mathbb{Z}^d$. We denote with $d\dx$ the Lebesgue measure on $I_0$. 
When we say that a function is {\it integrable on $I_0$} without any further precision, it is understood that this function is integrable with respect to $d\dx$. 

In the sequel, $C([0,+\infty),I_0)$ is the space of continuous functions defined on $[0,+\infty)$ with values in $I_0$ and 
 $(\dX_t\,;\, t\ge 0)$ is the canonical coordinate process on $C([0,+\infty),I_0)$. 

Let $x\in\mathbb{R}^d$ whose projection on $I_0$ we denote with $\dx$. 
Given a trajectory $(\dX_t\,;\, t\ge 0)$ in $C([0,+\infty),I_0)$ such that $\dX_0=\dx$, we let $(X_t\,;\, t\ge 0)$ be the $\mathbb{R}^d$-valued trajectory obtained by lifting $\dX$. That is $(X_t\,;\, t\ge 0)$ is the unique element in $C([0,+\infty),\mathbb{R}^d)$ satisfying $X_0=x$ and whose projection on $I_0$ coincides with $\dX_t$ for all times $t$. 

We shall consider the divergence-form operator $\dot{L}$ on $L^2(I_0; e^{-V(\dx)}d\dx)$, formally defined by:
\begin{displaymath}
\dot{L}f(\dx)=\frac 12 e^{V(\dot{x})}div(e^{-V(\dot{x})}\nabla f(\dot{x})).
\end{displaymath}

Ours first goal in this section is to prove that there exists a diffusion process associated to the operator $\dot L$ when $e^V$ and $e^{-V}$ are both integrable on $I_0$. In other words, we want to prove the existence of a Markov law $(P_{\dot x}\,;\, \dot{x}\in I_0)$ on $C([0, +\infty); I_0)$ with generator $\dot L$. Once this is done, we shall define the diffusion process in $\mathbb{R}^d$ by lifting the trajectory from the torus to $\mathbb{R}^d$. We first study the Dirichlet form associated with $\dot L$. 

Let $f$ and $g$ be a real-valued functions defined on $I_0$. 
For $i=1...d$, let $\partial_i f$ denote the the weak derivative of $f$ in the $i$-th direction. 
Let $f$ and $g$ be such that for any $i=1...d$, then $\partial_i f$ belongs to $L^2(I_0; e^{-V(\dx)}d\dx)$. We then define the bilinear forms  
\begin{equation}
\label{mouss}
\dxi(f,g):=\frac{1}{2}\int_{I_0}\nabla f(\dx)\cdot\nabla g(\dx)\, e^{-V(\dx)}d\dx,
\end{equation} 
and, if $f$ and $g$ are further assumed to belong to $L^2(I_0; e^{-V(\dx)}d\dx)$, 
\[\dxi_{1}(f,g):= \dxi(f,g)+\int_{I_0} f(\dx)g(\dx)\, e^{-V(\dx)}d\dx.\] 
More generally, for $\lambda>0$ and such functions $f$ anf $g$, let 
\[\dxi_{\lambda}(f,g):= \dxi(f,g)+\lambda\int_{I_0} f(\dx)g(\dx)\, e^{-V(\dx)}d\dx.\] 
Let ${\mathcal H}^1(I_0;e^{-V})$ be the set of functions in $L^2(I_0; e^{-V(\dx)}d\dx)$ with all derivatives $\partial_i f$ belonging to $L^2(I_0; e^{-V(\dx)}d\dx)$. 

\medskip

We recall the following definitions from \cite{FUK}. 

\textbf{Definitions}: 
A {\it Dirichlet form} $\dxi$ on $L^2(I_0; e^{-V(\dx)}d\dx)$ is a bilinear symmetric form on $L^2(I_0; e^{-V(\dx)}d\dx)$ with dense domain which is closed and Markovian. {\it Closed} means its domain
is complete with respect to the norm $\dxi_{1}(., .)$.\\
We say that a bilinear form is {\it Markovian} if, whenever $f\in Dom(\dxi)$, then 
$f':=(0\vee f)\wedge 1\in Dom(\dxi)$ and $\dxi(f',f')\leq \dxi(f,f)$.\\ 
Let $C(I_0)$ be the set of continuous functions defined on $I_0$ and  
let $C^\infty(I_0)$ be the set of smooth functions on $I_0$. 
A {\it core} of a bilinear symmetric form $\dxi$ is by definition a subset $\textsl{C}\subset Dom(\dxi)\bigcap C(I_0)$ such that $\textsl{C}$ is dense in $Dom(\dxi)$ with $\dxi_1$-norm and dense in $C(I_0)$ with uniform norm. A bilinear form that possesses a core  is called  {\it regular}.\\ 
A symmetric bilinear form $\dxi$ with domain $Dom(\dxi)$ is {\it closable} if for all sequence $(f_{n})$  in $Dom(\dxi)$ which goes to zero in $L^2(I_0; e^{-V(\dx)}d\dx)$ and such that  $(f_{n})$ is $\dxi$-Cauchy then  $\dxi(f_{n},f_{n})\rightarrow 0$.  A closable bilinear symmetric form has a smallest closed extension.\\ 
The Dirichlet form $\dxi$ is called {\it local} if whenever 
$f, g \in Dom(\dxi)$ are such that $supp(f)$ and $supp(g)$ are disjoints compact sets, then  $\dxi(f,g)=0$. Here $supp(f)$ and $supp(g)$  are the supports of the functions $f$ and $g$. 

\medskip 

We  have the following Proposition: 
\begin{prop} \label{prop:close} 
Assume that $e^V$ and $e^{-V}$ are integrable on $I_0$. 
The bilinear symmetric form $\dxi$ on ${\mathcal H}^1(I_0;e^{-V})$ is a  local Dirichlet form. 
\end{prop}

\begin{proof} 

The Markovian property is proved in \cite{PZ} page 36, lemma 3.2. The local property is obvious from the definition. 

Let $(f_{n})$ be a sequence in  ${\mathcal H}(I_0;e^{-V})$ which goes to zero in $L^2(I_0; e^{-V(\dx)}d\dx)$ and such that  $(f_{n})$ is $\dxi$-Cauchy.\\ 
Since $(f_n)$ is $\dxi$-Cauchy, we see that  $\nabla f_{n}$ is Cauchy in $L^2(I_0; e^{-V(\dx)}d\dx)$. Therefore $\nabla f_n$ converges to some limit $h$ in 
$L^2(I_0; e^{-V(\dx)}d\dx)$.\\
Since $\int_{I_0}e^{V(\dx)}d\dx<\infty $, then, for all $g \in C^\infty(I_0)$, we have 
\[c:=\left(\int_{I_0} \left(\vert g(\dx)\vert^2+\vert\nabla g(\dx)\vert^2\right)e^{V(\dx)}d\dx\right)^{\frac 12}<\infty,\] and, using the Cauchy-Schwarz inequality: 
\[
\begin{aligned}
\left|\int_{I_0}{g(\dx)\nabla f_{n}(\dx)d\dx-\int_{I_0} g(\dx)h(\dx) d\dx}\right|
&\leq\int_{I_0}{\left|g(\dx)\right|\left|\nabla f_{n}(\dx)- h(\dx) \right|e^{\frac{1}{2}V(\dx)-\frac{1}{2}V(\dx)}d\dx}\\
&\leq c\left(\int_{I_0} \left|\nabla f_{n}(\dx)-h(\dx)\right|^2 e^{-V(\dx)}d\dx\right)^{\frac{1}{2}}\rightarrow 0\\ \hbox{ when } n\rightarrow\infty.\\
\end{aligned}
\]

As $(f_n)$ converges to $0$ in  $L^2(I_0; e^{-V(\dx)}d\dx)$, we also have: 
\[
\begin{aligned}
\left|\int_{I_0} g(x)\nabla f_{n}(\dx)d\dx\right|&=\left|\int_{I_0} \nabla g(\dx) f_{n}(\dx)e^{\frac{1}{2}V(\dx)}e^{-\frac{1}{2}V(\dx)}d\dx\right|\\
&\leq c(\int_{I_0} \left|f_{n}(\dx)\right|^2 e^{-V(\dx)}d\dx)^{\frac{1}{2}} \rightarrow 0 \hbox{   when } n \rightarrow \infty.\\
\end{aligned}
\]

As a consequence of these two facts, we see that $\int_{I_0}{g(\dx)h(\dx)}d\dx = 0$ for all $g\in C^\infty(I_0)$. Therefore $h=0$  almost everywhere and 
 \[\dxi(f_{n},f_{n})\rightarrow 0\hbox{   when } n \rightarrow \infty.\] 

Thus we have proved that $\dxi$ is closed on   ${\mathcal H}^1(I_0;e^{-V})$.
\end{proof}

Let $H^1(I_0;e^{-V}):=\overline{C^\infty(I_0)}^{\dxi_1}$ be the completion of $C^\infty(I_0)$ 
with respect to the norm $\dxi_1$. 
Then $(\dxi, H^1(I_0;e^{-V}))$ is a regular local Dirichlet form. 

Following \cite{FUK}, chapter 1.5, we also define the {\it extended domain} $H^1_e(I_0)$: this is the set of 
measurable functions  $f$ on $I_0$, such that $\left| f\right|<\infty$ a.e and  there exists a $\dxi$-Cauchy sequence 
$(f_n)$ in $H^1(I_0;e^{-V})$ such that $\lim_{n\rightarrow\infty} f_n= f$ a.e. 


Since $(\dxi,H^1(I_0;e^{-V}))$  is a regular and local Dirichlet form, there exists a Markov law on $C([0,+\infty),I_0)$ whose Dirichlet form is $(\dxi,H^1(I_0;e^{-V}))$.  
This law is denoted with $(P_{\dx}\,;\,\dx\in I_0)$. It is uniquely defined for Lebesque almost all $\dx \in I_0$. The measure $e^{-V(\dx)}d\dx$ is reversible. 
The process thus defined is conservative and its generator, in the $L^2$ sense, is given by $\dL$.  Let $(E_{\dx}\,;\,\dx\in I_0)$ denote the expectation with respect to $P_{\dx}\,;\,\dx\in I_0$. 

Let $x\in\mathbb{R}^d$ and $\dx$ be its projection on $I_0$. We denote with $P_x$ the law of the lifting of the trajectory $(\dX_t\,;\,t\ge 0)$ to $\mathbb{R}^d$ under $P_{\dx}$. Then $P_x$ is a probability on $C([0,+\infty),I_0)$. 

\begin{rem} 
One may ask whether $H^1(I_0;e^{-V})={\mathcal H}^1(I_0;e^{-V})$. 
The answer is 'yes' if $V$ is bounded (in which case the operator $\dL$ is uniformly elliptic). 
A similar result holds in $\mathbb{R}^d$ when $V$ is $C^\infty(I_0)$  
(in which case the operator $\dL$ is hypo-elliptic). 
See \cite{FUK} chapter 3.3.\end{rem} 

\medskip 

In the sequel we will have to consider time-changed processes. We discuss this construction now. 

Consider a function $w$ defined on $I_0$ satisfying the following conditions:
\begin{equation}
\label{sedonnerunemesure}
\left\{\begin{array}{ll}
w &> 0  \hbox{ a.e on }I_0\\
w &\in L^1(I_0;  d\dx) \\
\end{array}\right.
\end{equation}
We use the notation $w(d\dx):=w(\dx)d\dx$ for the measure with density $w$ with respect to $d\dx$. 

The measure, $w(d\dx)$ is a Radon measure with full support and it charges no set of zero capacity. 
The positive continuous additive functional with Revuz measure $w(d\dx)$ is given by:

\begin{equation}\label{eq:pcaf} A_t :=\int_{0}^{t} w({\dX}_s)e^{V({\dX}_s)}ds.\end{equation}

We consider the symmetric bilinear form $\left(\widetilde{\dxi}, \widetilde{H}^1(I_0; w)\right)$  defined on $L^2(I_0; w(d\dx))$ by:

\begin{equation}
\label{processuschange}
\left\{\begin{array}{ll}
\begin{aligned}
\widetilde{H}^1(I_0;w)&= \{ \phi\in L^2(I_0; w(d\dx)): \exists f\in H^1_e(I_0) : f=\phi \hbox{ a.e} \}\\
\widetilde{\dxi}(\phi,\phi)&= \dxi(f,f).
\end{aligned}

\end{array}\right.
\end{equation}

Then Lemma 6.2.1 of \cite{FUK} ensures that $\left(\widetilde{\dxi}, \widetilde{H}^1(I_0;w)\right)$ is a Dirichlet form.  In view of the definition of $H^1_e(I_0)$, 
we remark that  the extended domain of $\widetilde{\dxi}$ coincides with the extended domain of $\dxi$. Note that $\left(\widetilde{\dxi}, \widetilde{H}^1(I_0;w)\right)$ 
admits $C^\infty(I_0)$ as a core, see Theorem 6.2.1 in \cite{FUK}, and that it is conservative.  
  
Let us now introduce the time-changed process ${\widetilde{\dX}}$ defined by:
\begin{equation}\label{eq:xtilde} \widetilde{\dX}_t:={\dX}_{A^{-1}_t}; \textrm{ where } A^{-1}_t:= \inf\left\{ s>0 : A_s> t\right\}.\end{equation} 
is the inverse of $A$. Note that $\widetilde{\dX}$ is also a strong Markov process with continuous paths, see theorem A.2.12 of \cite{FUK}. It admits the measure $w(d\dx)$ as a reversible measure. 

The next statement is a special case of  Theorem 6.2.1 of \cite{FUK}: 

\begin{prop}\label{prop:dirichtilde}
The Dirichlet form of the time-changed process $\widetilde{\dX}$ on $L^2(I_0;w(d\dx))$ is given by $(\widetilde{\dxi}, \widetilde{H}^1(I_0;w))$. 
\end{prop} 

\medskip

\section{Sobolev Inequality: proof of Theorem \ref{th:poinc}.}
\label{sec:sobolev}

The proof of Theorem \ref{th:poinc} uses many tools from harmonic analysis. In particular the function $w$ that appears in Theorem \ref{th:poinc} is expressed as a  Hardy-Littlewood 
maximal function, see below.  

We start recalling the results we shall need from real harmonic analysis. 
We refer to the book of A. Torchinsky \cite{AT} where all the material below can be found.  

We recall that $I_0$ is the unit torus $\mathbb{R}^d/\mathbb{Z}^d$; $d\dx$ is the Lebesgue measure. 
We use the notation $\vert I\vert$ for the Lebesgue measure of a measurable subset $I\subseteq I_0$. 

\textbf{Definitions} 
 
\textbf{1)} Let $f$ be a measurable function on $I_0$. We assume that $f\in L^1(I_0; d\dx)$.\\
The Hardy-Littlewood maximal function is defined as: 
\begin{displaymath}
M(f)(\dx)=\sup_{I\subseteq I_0: \dx\in I}\frac{1}{\left|I\right|}\int_I\left|f(\dy)\right| d\dy, 
\end{displaymath}
where the $I'$s are open cubes containing $\dx$. Observe that the function $M(f)$ is non-negative and measurable.

\textbf{2)} \textbf{$A_p$} \textbf{condition:} let $p\ge1$. We say that a non-negative function $w\in L^1(I_0; d\dx)$ verifies the $A_p(I_0)$  condition, and we write $w\in A_p(I_0)$, if there exists a constant $c$ such that for all cube $I\subseteq I_0$:
\begin{displaymath}
 \qquad \qquad\qquad  \left(\frac{1}{\left|I\right|}\int_I w(\dy)d\dy\right)\left(\frac{1}{\left|I\right|}\int_I w(\dy)^{\frac{-1}{p-1}}dy\right)^{p-1}\leq c, \qquad if\qquad 1<p<\infty.
\end{displaymath}
\begin{displaymath}
\frac{1}{\left|I\right|}\int_I w(\dy) d\dy\leq c\, (\inf_I w)\qquad if\qquad  p=1.
\end{displaymath}

\textbf{3)} \textbf{$A_\infty$} \textbf{condition:} we say that $w$ verifies the $A_\infty(I_0)$ condition and we write $w\in A_\infty(I_0)$, if for each $0<\epsilon<1$ there corresponds $0<\delta<1$ so that for all measurable subset $E$ of $I$ we have $\int_E w(\dy) d\dy <\epsilon\int_{I} w(\dy)d\dy$ whenever $\left|E\right|<\delta \left|I\right|$. 
One proves that  
\begin{equation}
\label{yoyo}
A_\infty=\bigcup_{p>1}A_p,
\end{equation}  
see remark 8.10 in chapter 9 of \cite{AT}. 

\textbf{4)}  \textbf{Proposition 3.3 of \cite{AT} (Coifman and Rochberg)}\\
Let $f\in L^1(I_0; d\dx)$. Then, for each $0\leq \epsilon <1$, we have 
$\left(M(f)\right)^\epsilon \in A_1(I_0)$.

\textbf{5)} Let us define also  the set 
\[\begin{aligned} 
A_{p,s,\frac{1}{d}}(I_0):= \big\{(w ,v)&\in L_+^1( I_0; d\dx)\,: \exists c >0: \\
&\forall I\subseteq I_0; \left(\int_{I}w(\dx)d\dx\right)^\frac{1}{s}\left(\int_I v(\dx)^{\frac{-1}{p-1}} d\dx\right)^{\frac{p-1}{p}}\leq c \left|I\right|^{1-\frac{1}{d}}\big\}.
\end{aligned}\]   

We shall use the following \\ 
\textbf{Theorem 4.8 of \cite{AT}: (Sobolev's embedding theorem)}\\
Let $1<p< \infty$ and $s$ be such that $\frac{1}{p}-\frac{1}{d}\leq \frac{1}{s}< \frac {1}{p}$. Let $w\in A_\infty(I_0)$ and $(w,v) \in A_{p,s,\frac{1}{d}}(I_0)$. Then for any $q$ such that $p\leq q< s$, there exists a constant $c$ such that: 
\begin{displaymath}
\label{sob}
\left(\int_{I_0}\left|f(\dx)\right|^q w(\dx)d\dx\right)^{\frac{1}{q}} \leq c \left(\int_{I_0} \left|\nabla f(\dx)\right|^p  v(\dx)d\dx\right)^{\frac{1}{p}}, 
\end{displaymath} 
for every function $f$ defined on $I_0$, centered  and $C^1$ there. 

\medskip

Let us now prove Theorem \ref{th:poinc}.

We let 
\[w(\dx)=M(e^{V})(\dx)^{-1},\]
and check that this function $w$ satisfies all the properties in Theorem \ref{th:poinc}. 

First observe that since $e^{V}\in L^1(I_0; d\dx)$, then $w^{-1}=M(e^V)$ belongs to the weak $L^1(I_0; d\dx)$ space and therefore $M(e^V)< \infty$ a.e.  and $w>0$ a.e. 
Also $M(e^V)$ is bounded from below by $\int_{I_0} e^{V(\dy)} d\dy$ and therefore $w$ is bounded by $\left(\int_{I_0} e^{V(\dy)} d\dy\right)^{-1}$.

We shall apply Theorem 4.8 of \cite{AT} with $v(\dx)=e^{-V(\dx)}$ and $p=2$. In order to do so, it is sufficient to verify that 
$w\in A_{\infty}(I_0)$ and $(w,v) \in A_{2,s,\frac{1}{d}}(I_0)$ for some $s>2$. 

We first prove that $w\in A_\infty(I_0)$: 
the result of Coifman and Rochberg quoted in point 4) above  implies that $\frac{1}{\sqrt{w}}= M(e^V)^{\frac{1}{2}}\in A_1$. 
This implies that, for all $I$,  \[\frac{1}{\left| I \right|}\int_I  \frac 1{\sqrt{w(\dy)}}  d\dy \leq  c \left(\inf_I \frac 1{\sqrt{w}}\right),\] 
for some constant $c$.  
Therefore  
\[
\begin{aligned}
(\frac{1}{\left| I\right|}\int_{I} w(\dy)d\dy)(\frac{1}{\left|I\right|}\int_{I}\frac{1}{\sqrt{w(\dy)}} d\dy)^2 
&\leq c^2 \left(\frac{1}{\left| I\right|}\int_{I} w(\dy)d\dy\right)\inf_I \frac 1{w} \\
&=c^2 \frac{1}{\left| I\right|}\int_{I} \frac{w(\dy)}{\sup_I w} d\dy\leq c^2. 
\end{aligned}
\]
Therefore 
$w \in A_3(I_0)$ and, using remark  (\ref{yoyo}), $w\in  A_\infty(I_0)$. 

Let us now check that there exists $s> 2$  such that $(w,v) \in A_{2,s,\frac{1}{d}}(I_0)$.\\ 
By definition of the maximal function,  we know that for all $I\subseteq I_0$ and for all $\dx\in I$, then 
\[ w(\dx)\leq \left|I\right|\left(\int_I e^{V(\dy)} d\dy\right)^{-1}.\] Therefore  
\begin{equation} \label{eq:w} \left(\int_I w(\dy) d\dy \right)^{\frac{1}{s}}\left(\int_I e^{V(\dy)}d\dy\right)^{\frac{1}{2}}\leq \left|I\right|^{\frac{2}{s}}\left(\int_I e^{V(\dy)}d\dy\right)^{\frac{1}{2}-\frac{1}{s}}.\end{equation} 
We choose $s=\frac{2d}{d-1}$ and the following verifications are easy:
\begin{displaymath}
\left\{\begin{array}{ll}
{1}/{2}-{1}/{d}\leq{1}/{s}< {1}/{2},\\
{1}/{2}-{1}/{s}>0,\\
{2}/{s}= 1-{1}/{d},
\end{array}\right.
\end{displaymath} 
and it follows from (\ref{eq:w}) that 
\[\left(\int_I w(\dy) d\dy \right)^{\frac{1}{s}}\left(\int_I e^{V(\dy)}d\dy\right)^{\frac{1}{2}}\leq \left(\int_{I_0} e^{V(\dy)}d\dy\right)^{\frac{1}{2}-\frac{1}{s}} \left|I\right|^{1-\frac{1}{d}}.\] 
Thus we checked the $A_{2,s,\frac{1}{d}}(I_0)$ condition. 
 
Now Theorem 4.8 of \cite{AT} implies Theorem \ref{th:poinc} for any choice of $r^*$ such that 
$2<r^*<s=2d/(d-1)$. \qed

\medskip

\begin{rem} \label{rem:classicSob}

Here is an elementary proof of Theorem \ref{th:poinc} when $e^V$ belongs to $L^r(I_0,\,d\dx)$ for some $r>d/2$. 

The usual Sobolev inequality states that for all $p\in[1,d)$, then 
\[\left(\int_{I_0}\left|f(\dx)\right|^q d\dx\right)^{\frac{1}{q}} \leq c \left(\int_{I_0} \left|\nabla f(\dx)\right|^p d\dx\right)^{\frac{1}{p}}\,,\] 
for every function $f$ defined on $I_0$, centered  and $C^1$ there with $q=pd/(d-p)$. 

Choose $p\in[1,2)$.  Applying H\"older's inequality with parameters $a=2/p$ and $b=2/(2-p)$, we get that 
\[\left(\int_{I_0}\left|f(\dx)\right|^q d\dx\right)^{\frac{1}{q}} \leq c \left(\int_{I_0} \left|\nabla f(\dx)\right|^2  e^{-V(\dx)} d\dx\right)^{\frac{1}{2}}
 \left(\int_{I_0} e^{\frac p{2-p}V(\dx)} d\dx \right)^{\frac{2-p}{2p}}\,.
\] 
Letting $p$ approach $2d/(d+2)$ with $p>2d/(d+2)$, we then get the inequality (\ref{eq:sob}) with constant function $w=1$ and provided that 
$e^V$ belongs to $L^r(I_0,\,d\dx)$ for some $r>d/2$. 

Although this method seems to work only if $e^V$ belongs to $L^r(I_0,\,d\dx)$ for some $r>d/2$, it has the advantage of providing an explicit 
and simple expression of the constant in terms of $V$. 

\end{rem} 

\begin{rem} \label{rem:ADS} 

One may compare our approach with the one used in \cite{ADS}. 

We recall that \cite{ADS} proves a quenched invariance principle for random walks with random conductances under $L^p$ integrability conditions on the conductances and their inverses 
where $p$ is much larger than $1$. 

The proof of \cite{ADS} also relies on Sobolev inequalities. Since the environment may not be periodic, there is no finite scale that controls everything. Therefore, rather than one single Sobolev inequality, one needs a sequence of Sobolev inequalities on a growing family of balls centered at the origin. In \cite{ADS}, these are obtained from the classical (discrete) Sobolev embedding as in Remark \ref{rem:classicSob}. This explains why the integrability condition in \cite{ADS} is not optimal. On the other hand, combining our technics with those of \cite{ADS} in the random environment setting would require some information on the constant appearing in our Theorem \ref{th:poinc}. 

\end{rem}  

From now on, we assume that $e^V$ and $e^{-V}$ are integrable on $I_0$. We choose the function $w$ given by Theorem \ref{th:poinc}. 
We recall that the process $\widetilde \dX$ is obtained from the process $\dot X$ through the time-change with additive functional 
\[
A_t =\int_{0}^{t} w({\dX}_s)e^{V({\dX}_s)}ds,
\]
see (\ref{eq:xtilde}). The Dirichlet form of the process $\widetilde \dX$ is given by the bilinear form $\widetilde\dxi$ defined on $L^2(I_0; w(d\dx)$ with extended domain described in Proposition \ref{prop:dirichtilde}. 

Since $C^1$ functions are dense in the domain of $\widetilde \dxi$, it follows that equation (\ref{eq:sob}) is true for any function $f$ in $\widetilde{H}^1(I_0;w)$. 

Let $({\widetilde \dP}_t; t\ge 0)$ be the semi-group generated by $\widetilde \dX$. By construction, $({\widetilde \dP}_t; t\ge 0)$ is a symmetric strongly continuous semi-group acting on $L^2(I_0; w(d\dx))$. 
It is related to the process ${\widetilde \dX}$ through the formula 
\[ {\widetilde \dP}_tf(\dx)=E_{\dx}[f({\widetilde \dX}_t)],\] 
for almost any $\dx\in I_0$, any time $t$ and any measurable function $f\in L^2(I_0; w(d\dx))$. 

As a consequence of Theorem \ref{th:poinc}, we have the following 

\begin{cor}
\label{densbornee} 
For all positive time $t$, for almost every $\dx\in I_0$, the law of ${\widetilde \dX}_t$ under $P_{\dx}$ has a density with respect to the measure $w(d\dot x)$, say $({\widetilde{\dot p}}_t(\dx,\dy); \dy\in I_0)$. 
The function $(\dx,\dy)\rightarrow \widetilde{\dot p}_t(\dx,\dy)$ is almost everywhere bounded on $I_0\times I_0$. 
\end{cor}
  
\begin{proof}

The proof follows a classical argument that can be found in the book \cite{D} or the papers \cite{VN} and \cite {CKS} for instance.\\ 
We only sketch it here. 

In the proof below, the value of the constant $c$ may vary from line to line.

Choose $r^*$ from Theorem \ref{th:poinc} and let $p=r^*/2$.  
Equation (\ref{eq:sob}) then reads: for all $C^1$ and centered function $f$, then 
\begin{equation}
\label{mou}
\left(\int_{I_0}\left|f(\dx)\right|^{2p} w(d\dx) \right)^{\frac{1}{p}} \leq c\int_{I_0} \left|\nabla f(\dx)\right|^2 e^{-V(\dx)}d\dx.
\end{equation}
Using first H\"older's inequality with parameters $2p-1$ and $(2p-1)/(2p-2)$ and then (\ref{mou}) we deduce that 
\begin{equation}
\label{mo}
\begin{aligned} 
&\int_{I_0}f^2(\dx)w(d\dx) =\int_{I_0}\left|f(\dx)\right|^{2p/(2p-1)}\left|f(\dx)\right|^{(2p-2)/(2p-1)}w(d\dx)  \\ 
&\leq \left(\int_{I_0}\left|f(\dx)\right|^{2p}w(d\dx)\right)^{1/({2p-1})}\left(\int_{I_0}\left|f(\dx)\right|w(d\dx)\right)^{({2p-2})/({2p-1})}\\ 
&\leq c \left(\int_{I_0}\left|\nabla f(\dx)\right|^2 e^{-V(\dx)} d\dx\right)^{{p}/({2p-1})}\left(\int_{I_0}\left|f(\dx)\right|w(d\dx)\right)^{({2p-2})/({2p-1})}. 
\end{aligned} 
\end{equation}

Using the density of $C^1$ functions,  inequality (\ref{mo}) can be extended for all centered functions $f\in \widetilde{H}^1(I_0;w)$. Then  
$\int_{I_0}\left|\nabla f(\dx)\right|^2 e^{-V(\dx)} d\dx=2\widetilde{\dot{\xi}}(f,f)$. 

Let $f\in L^2(I_0; w(dx))$ and set $f_t:=\widetilde {\dP}_tf$. Assume that $f$ is centered. Then so is $f_t$ for any $t$.\\ 
Let $v(t):=\left(\int_{I_0}\vert f(\dx)\vert w(d\dx)\right)^{-2} \int_{I_0} f_t(\dx)^2 w(d\dx)$.

On the one hand, the function $v$ satisfies 
\[v'(t)=-2\left(\int_{I_0}\vert f(\dx)\vert w(d\dx)\right)^{-2}\widetilde{\dot{\xi}}(f_t,f_t).\] 
Therefore, using (\ref{mo}), we have 
\[v(t)\leq \left(-c v^{'}(t)\right)^{\alpha}\qquad \textrm{\, with\, } \alpha=\frac{p}{2p-1}.\]
(We used the fact that \[
\int_{I_0}\left| f_t(\dx)\right| w(d\dx)\leq \int_{I_0}\widetilde {\dP}_t\left| f(\dx)\right| w(d\dx)=\int_{I_0}\left|f(\dx)\right|w(d\dx).) 
\]

From this differential inequality, we deduce that $v(t)$ is bounded by a constant, say $c(t)$, independently of $f$ and therefore 
\[ \int_{I_0} \left(\widetilde {\dP}_tf\right)^2 w(d\dx)\leq c(t)  \left(\int_{I_0}\vert f(\dx)\vert w(d\dx)\right)^{2}.\] 

The duality property gives:$$\left\|\widetilde {\dP}_tf\right\|_{L^\infty}=\sup\left\{\frac{\left|\int_{I_0}g(\widetilde{\dP}_tf ) w(dx)\right|}{\left\|g\right\|_{L^1(I_0; w(d\dx))}} ; g\in L^1(I_0; w(d\dx)), g\neq 0 \right\}.$$ 
Thus  
we have using H\"older's inequality again:
\[\left\|\widetilde {\dP}_tf\right\|_{L^\infty}\leq \sqrt{c(t)}\left(\int_{I_0}\left|f(\dx)\right|^2 w(d\dx)\right)^{\frac{1}{2}} \qquad \forall t> 0. \]

As a consequence 
\[\begin{aligned} &\left\|\widetilde{\dP}_{t} f\right\|_{L^\infty}= \left\|\widetilde{\dP}_{t/2}(\widetilde {\dP}_{t/2} f\right)\|_{L^\infty}\\ 
&\leq \sqrt{c(t/2)}\left\|\widetilde{\dP}_{t/2} f\right\|_{L^2(I_0; w(d\dx))}\leq \sqrt{c(t/2)} \sqrt{c(t/2)}\left\|f\right\|_{L^1(I_0; w(dx))}.\end{aligned}\]

We deduce that:
\begin{displaymath}
\left\|\widetilde {\dP}_t f \right\|_{L^\infty(I_0)}\leq c(t/2)\left\|f \right\|_{L^1(I_0; w(d\dx))} \qquad \forall t>0.
\end{displaymath}
This inequality extends to all non-negative functions $f$. 

By taking $f=\mathbf{1}_A $,  with A any  Borelian contained in $I_0$ we deduce that the semi-group $\widetilde {\dP}$ is absolutely continuous with respect to the measure $w(d\dot x)$ with a density  bounded by $c(t/2)$. 
\end{proof}

\section{Construction of the corrector}
\label{sec:corr} 

We already defined the process $(X_t\,;\, t\ge 0)$ as the lifting of $(\dX_t\,;\, t\ge 0)$. 
Recall that the process $(\widetilde{\dX}_t\,;\, t\ge0)$ is obtained from $\dX$ by the time change $A$ from equation (\ref{eq:xtilde}). 
We similarly introduce the process $(\widetilde{X}_t\,;\,t\ge 0)$ as the time-change of $X$ through the additive functional $A$. Note that the 
projection on $I_0$ of the trajectory of $\widetilde{X}$ is then $(\widetilde{\dX}_t\,;\, t\ge0)$. 

\medskip

In this section, we prove the existence of a corrector to the process $\widetilde{X}$, i.e. we construct a function $v$, defined on $I_0$, such that 
$ \widetilde{M}_t:=\widetilde{X}_t + v(\widetilde{\dot{X}}_t)$  is a continuous martingale under $P_x$ for almost all $x\in \mathbb{R}^d$.

We use the construction of the Dirichlet form  $\dxi$ from part \ref{sec:existence}, where the function $w$ is the one given by Theorem \ref{th:poinc}. 
In particular recall that $H^1_e(I_0)$ is the extended domain of $\widetilde{\dxi}$. Observe that the Sobolev inequality (\ref{eq:sob}) implies that functions in $H^1_e(I_0)$ 
are also in $L^{r^*}(I_0; w(d\dx))$ and therefore in $L^1(I_0; w(d\dx))$. \\ 
We call $H^1_{o,e}(I_0)$ the quotient space obtained by identifying 
functions in  $H^1_e(I_0)$ when they differ by a constant. Equivalently $H^1_{o,e}(I_0)$ is the sub-space of centered functions 
in $H^1_e(I_0)$. \\
To start the construction of the corrector,  we need the following proposition. 
\begin {prop}
\label{hilbert} 
$(H^1_{o,e}(I_0), \widetilde{\dxi})$ is a Hilbert space. 
\end{prop}

\begin{proof} 
The proposition follows from the Poincar\'e inequality 
\begin{equation}\label{eq:poinc}
\int_{I_0}\left|f(\dx)\right|^{2} w(d\dx)\leq c \int_{I_0}\left|\nabla f(\dx)\right|^2 e^{-V(\dx)} d\dx, 
\end{equation}
which is itself a consequence of (\ref{eq:sob}) and H\"older's inequality. 

On the one hand, (\ref{eq:poinc}) implies that $\widetilde{\dxi}$ is a norm on $H^1_{o,e}(I_0)$ and it is equivalent to $\widetilde{\dxi}_1$. \\ 
Since $\widetilde{H}^1(I_0; w)$ is complete with respect to $\widetilde{\dxi}_1$, and because the condition of being centered is closed in $\widetilde{H}^1(I_0; w)$, we get 
that $(H^1_{o,e}(I_0), \widetilde{\dxi})$ is complete. 
\end{proof} 

\begin{rem} \label{rem:domaine2} 
Observe, as above, that functions in $H^1_e(I_0)$ are also in $L^2(I_0; w(dx))$. 

Therefore $H^1_e(I_0)=\widetilde{H}^1(I_0; w)$. 
\end{rem}

 Construction of the corrector: for $i=1...d$, consider the expression: 
 \[ L_i: f\mapsto -\frac 12 \int_{I_0}  \partial_i f(x)\, e^{-V(x)} dx\,.\] 
Then $L_i$ is a continuous linear map on $(H^1_{o,e}(I_0), \widetilde{\dxi})$. \\ 
We identify $H^1_{o,e}(I_0)$ and its dual. Thus, there exists a unique $v_i$ in $H^1_{o,e}(I_0)$ such that:
\begin{equation}
\label{harmonique}
\begin{aligned} 
-\frac 12\int_{I_0} \partial_i f(x)\, e^{-V(x)} dx&= \frac 12\int_{I_0} \nabla v_i\cdot\nabla f\, e^{-V(x)}dx\\
&= \widetilde{\dxi}( v_i, f),
\end{aligned}
\end{equation}
for all $f\in H^1_{o,e}(I_0)$. 

The function $v_i$ is called the {\it corrector} in the direction $i$. We may also consider the vector-valued corrector 
$v:=(v_1,...,v_d):I_0\mapsto \mathbb{R}^d$. 
We also define the function $u=(u_1,...,u_d)$ from $\mathbb{R}^d$ to $\mathbb{R}^d$ by  
$u(x)=x+v(\dx)$ (where $\dx$ is the projection of $x$ on $I_0$). 

\begin{prop}\label{mart} 
The process $(\widetilde{M}_t:=u(\widetilde{X}_t)=\widetilde{X}_t + v(\widetilde{\dX}_t)\,;\,t\ge 0)$  is a continuous martingale under $P_x$ for almost all $x\in \mathbb{R}^d$ and satisfies 
\begin{equation}\label{eq:decomp} 
\langle \widetilde{M}\rangle_t=\int_0^t \frac {e^{-V(\widetilde{\dX}_s)}}{w(\widetilde{\dX}_s)}((\delta+\nabla v)(\delta+\nabla v))(\widetilde{\dX}_t)\, dt\,,\end{equation} 
where $((\delta+\nabla v)(\delta+\nabla v))(.)$ is the matrix with $(i,j)$ entry given by $(\delta_i+\nabla v_i(.))\cdot(\delta_j+\nabla v_j(.))$ and 
$\delta_i$ is the unit vector in direction $i$. 
\end{prop} 

\begin{proof} 

We recall from \cite{FUK}, chapter 5, that for all functions $f\in \widetilde{H}^1(I_0;w)$, the process 
$t\to f(\widetilde{\dX}_t)$ has a unique It\^o-Fukushima decomposition under $P_{\dx}$, for almost every $\dx$, as a sum of two terms: 
\begin{equation}\label{eq:itof}f(\widetilde{\dX}_t)- f(\widetilde{\dX}_0)=M^f_t+N^f_t\,,\end{equation} 
where $M^f$ is a continuous martingale additive functional and $N^f$ is a functional of zero energy. 
Besides, for $f$ and $g$ in $\widetilde{H}^1(I_0;w)$, one has the following expression for the square bracket: 
\begin{equation}\label{eq:quadvar} 
\langle M^f,M^g\rangle_t=\int_0^t \frac {e^{-V(\widetilde{\dX}_s)}}{w(\widetilde{\dX}_s)} \nabla f(\widetilde{\dX}_s)\cdot\nabla g(\widetilde{\dX}_s)\,ds.
\end{equation}
See in particular example 5.2.1 and formula (5.2.46) in \cite{FUK}. 

These formulas do not immediately yield a decomposition for the process $\widetilde{M}$. Indeed, we could directly apply the  It\^o-Fukushima decomposition to the function $v$ 
which belongs to $\widetilde{H}^1(I_0;w)$ , but, although the process $\widetilde X$ is also an additive functional of $\widetilde \dX$, it is not of the form (\ref{eq:itof}). 
In order to deal  this difficulty, we rely on a localization argument. 

Let $\dx\in I_0$ and choose $x\in\mathbb{R}^d$  whose projection on $I_0$ is $\dx$. 
Let $J_0$ be a closed cube in $I_0$ centered at $\dx$. We identify $J_0$ with a closed cube in $\mathbb{R}^d$ centered at $x$, say $J_1$, and let $\phi:J_0\to J_1$ be the identification map. 

We will denote with $(^c\widetilde{\dX}_t\,;\,t\ge0)$ the process obtained by reflecting $\widetilde{\dX}$ on the boundary of $J_0$. The construction of $^c\widetilde{\dX}$ mimics the construction 
of $\widetilde{\dX}$ in part \ref{sec:existence} except that we consider the bilinear form (\ref{mouss}) on smooth functions with support in $J_0$. Let $^c\widetilde{\dxi}$ be the Dirichlet form of the process $^c\widetilde{\dX}$. 

Let $\tau$ be the hitting time of the boundary of $J_0$. Note that the two processes $^c\widetilde{\dX}_t$ and $\widetilde{\dX}_t$ coincide in law until time $\tau$. Besides, the two processes $\widetilde X$ and $\phi(^c\widetilde{\dX})$ 
also coincide until time $\tau$. Thus we get that 
\begin{equation}\label{eq:local} 
u(\widetilde{X}_t)-u(\widetilde{X}_0)=(v+\phi)(^c\widetilde{\dX}_t)-(v+\phi)(^c\widetilde{\dX}_0)\,, 
\end{equation} for times $t<\tau$ (in the sense that these two processes have the same law). 

Now observe that the functions $v$ and $\phi$ both belong to the domain of the Dirichlet form $^c\widetilde{\dxi}$.  Thus the process $(v+\phi)(^c\widetilde{\dX})$ admits an It\^o-Fukushima decomposition 
as \[(v+\phi)(^c\widetilde{\dX}_t)-(v+\phi)(^c\widetilde{\dX}_0)=M^{(0)}_t+N^{(0)}_t\,.\] 
On the one hand, the function $\partial_i \phi$ is constant and equals the unit vector in direction $i$. On the other hand, the function $v$ satisfies equation (\ref{harmonique}). Thus we get that 
$^c\widetilde{\dxi}(f,v+\phi)=0$ for all smooth functions $f$ supported in the interior of $J_0$. In other words, the function $v+\phi$ is harmonic for the process $^c\widetilde{\dX}$ killed at time $\tau$. 
It implies that the process $(u(^c\widetilde{\dX}_t)-u(^c\widetilde{\dX}_0)\,;\,0\le t<\tau)$ is a local martingale and $N^{(0)}_t=0$ for all times $t<\tau$. Using (\ref{eq:local}), we conclude that 
the process $(u(\widetilde{\dX}_t)-u(\widetilde{\dX}_0)\,;\,0\le t<\tau)$ is a local martingale. 

In order to prove that $(u(\widetilde{\dX}_t)-u(\widetilde{\dX}_0)\,;\,0\le t)$ is a local martingale for all times, one iterates this reasoning using the Markov property. The computation of the bracket follows from 
formula (\ref{eq:quadvar}). 

\end{proof} 

\section{Homogenization results: proof of theorem 1}
\label{sec:hom}


We show the invariance principle for $\widetilde{X}$ and deduce the invariance principle for $X$ using the relation (\ref{eq:xtilde}).

\textbf{A/ } \textbf{ Invariance Principle for } \textbf{$\widetilde{X}$} 

Let $\widetilde{X}^{(\epsilon)}_t:= \epsilon \widetilde{X}_{t/\epsilon^{2}}$ and 
$\widetilde{\dX}^{(\epsilon)}_t:= \epsilon \widetilde{\dX}_{t/\epsilon^{2}}$. 

\begin{prop} \label{pixtilde}
There exists a positive symmetric non-degenerate matrix $\sigma$  such that for almost all x $\in \mathbb{R}^d$, under $P_x$, the family of processes $(\widetilde{X}^{(\epsilon)}\,;\,\epsilon>0)$ converges in distribution, as $\epsilon$ tends to zero,  
towards the law of a Brownian motion with variance $\sigma$. 
\end{prop} 

The proof of Proposition \ref{pixtilde} is in two steps: 

\medskip  

\textbf{First step:} invariance principle for the martingale part. 

We define $u_i^\epsilon(x)=\epsilon u_i(\frac{x}{\epsilon})$ and let  
\[\widetilde{M}_t^{i, \epsilon}:=u_i^\epsilon(\widetilde{X}^\epsilon_t)-u_i^\epsilon(\widetilde{X}^\epsilon_{0}),\] 
\[\widetilde{M}_t^{\epsilon}:=(\widetilde{M}^{1,\epsilon},...,\widetilde{M}^{d,\epsilon}).\]

\begin{lemma}\label{lem:pimart} 
There exists a positive symmetric non-degenerate matrix $\sigma$  such that for almost all x $\in \mathbb{R}^d$, under $P_x$, the family of processes $(\widetilde{M}^\epsilon\,;\,\epsilon>0)$ converges in distribution, as $\epsilon$ tends to zero,  
towards the law of a Brownian motion with covariance matrix $\sigma$. 
\end{lemma} 

\begin{proof}

We will need the invariance principle for continuous martingales. 
For the reader's convenience, we provide here the formulation of theorem 5.1 of \cite{IH}.

\textbf{Theorem 5.1 of \cite{IH} (Helland 1982) }\\
Let $m^\epsilon$ be a family of continuous real-valued martingales with quadratic variation processes $<m^\epsilon>$ satisfying the following condition:\\
$(i)$  there exists a real number $a>0$ such that for any $t>0$, as $\epsilon$ tends to zero,  then $<m^\epsilon>_t$ converges in probability to $a t$. \\ 
Then, as $\epsilon$ tends to zero, the sequence of processes $m^\epsilon(.)$ converges in law in the uniform topology to a Brownian motion with covariance $a$. 

\medskip 

Let  $\sigma$ be the matrix with entries given by $$(\sigma)_{i,j}:= \int_{I_0} (\delta_i+\nabla v_i(\dx))\cdot (\delta_j+\nabla v_j(\dx)) e^{-V(\dx)} d\dx.$$  
Note that, by construction, $\nabla v_j$ belongs to $L^2(I_0;e^{-V(\dx)}d\dx)$. 

In view of Proposition \ref{mart}, we know that   $\widetilde{M}_t^{i,\epsilon}$ is a square integrable martingale which quadratic variation 
\[ \int_{0}^{t}\left|\delta_i+\nabla{v_i}\right|^2(\frac{\widetilde{X}^{(\epsilon)}_s}{\epsilon}) \left(\frac{e^{-V}}{w}\right) (\frac{\widetilde{X}^{(\epsilon)}_s}{\epsilon})ds 
=\int_{0}^{t}\left|\delta_i+\nabla{v_i}\right|^2(\frac{\widetilde{\dot{X}}^{(\epsilon)}_s}{\epsilon}) \left(\frac{e^{-V}}{w}\right) (\frac{\widetilde{\dot{X}}^{(\epsilon)}_s}{\epsilon})ds, \]
 because $V$ is periodic, $\left|\delta_i+\nabla v_i\right|^2$ is periodic and $w$ is also periodic. 
 
More generally, for any vector $e\in\mathbb{R}^d$, then $e\cdot \widetilde{M}^\epsilon_t:=\sum_i e_i \widetilde{M}_t^{i,\epsilon}$ is a square integrable martingale with bracket 
\[  \langle e\cdot \widetilde{M}^\epsilon\rangle_t=\int_{0}^{t}\left(\sum_i e_i (\delta_i+\nabla{v_i})\right)^2(\frac{\widetilde{\dot{X}}^{(\epsilon)}_s}{\epsilon}) \left(\frac{e^{-V}}{w}\right) (\frac{\widetilde{\dot{X}}^{(\epsilon)}_s}{\epsilon})ds   \]

By the ergodic Theorem for $\widetilde{\dot{X}}$, for all $t\geq 0$:
 \[
\langle e\cdot \widetilde{M}^\epsilon\rangle_t\longrightarrow _{\epsilon\rightarrow 0} t.\int_{I_0}\left(\sum_i e_i (\delta_i+\nabla{v_i}(\dx))\right)^2 e^{-V(\dx)}d\dx\textrm { almost surely.} 
\]

Theorem 5.1 of \cite{IH}, as recalled above,  gives  the invariance principle for the martingales $\left(e\cdot \widetilde{M}_t^{\epsilon}; t\geq 0\right)$ with asymptotic variance $e\cdot \sigma e$. Since this is true for all direction $e$, we deduce the invariance principle for $M^\epsilon$ itself.

\end{proof}


\textbf{Second step:} convergence of the corrector. 

We have to show that the corrector part goes to zero in $P_x$ probability for almost all $x\in \mathbb{R}^d$. For that, it suffices to prove the following equality:  
\begin{equation}
\label{correctortendsverszerouniformement}
\forall \eta >0,\\
\limsup_{\epsilon\downarrow 0}P_x\left(\sup_{0\leq t\leq 1}\left|\epsilon v_i(\frac{\widetilde{\dot{X}}^{(\epsilon)}_t}{\epsilon})\right|> \eta \right)= 0.
\end{equation} 

\pagebreak

Observe that (\ref{correctortendsverszerouniformement}) implies that,  for all $T$ and for all $\eta>0$, 
\begin{equation}
\label{verzero}
\begin{aligned}
&\limsup_{\epsilon\downarrow 0}P_x\left(\sup_{0\leq t\leq T}\left|\epsilon v_i(\frac{\widetilde{\dot{X}}^{(\epsilon)}_t}{\epsilon})\right|> \eta \right)\\ 
=&\limsup_{\epsilon\downarrow 0}P_x\left(\sup_{0\leq t\leq 1}\left|\epsilon v_i(\frac{\widetilde{\dot{X}}^{(\epsilon)}_t}{\epsilon})\right|> \frac{\eta}{\sqrt{T}} \right)= 0.
\end{aligned}
\end{equation}

We have $\widetilde{X}^{(\epsilon)}=\widetilde{M}^\epsilon-v(\widetilde{\dX}^{(\epsilon)})$. 
Combining (\ref{verzero}) with Lemma \ref{lem:pimart} yields Proposition \ref{pixtilde}. 

\qed 

\medskip 

Now, let us prove (\ref{correctortendsverszerouniformement}). We have 
\[
\begin{aligned}
P_x\left(\sup_{0\leq t\leq 1}\left|\epsilon v_i(\frac{\widetilde{\dot{X}}^{(\epsilon)}_t}{\epsilon})\right|> \eta \right) &\leq P_x\left(\sup_{0\leq t\leq \epsilon^2}\left|\epsilon v_i(\frac{\widetilde{\dot{X}}^{(\epsilon)}_t}{\epsilon})\right|>\eta\right) (:=I)\\
&+P_x\left(\sup_{\epsilon^2\leq t\leq 1}\left| \epsilon v_i(\frac{\widetilde{\dot{X}}^{(\epsilon)}_t}{\epsilon})\right|>\eta\right)(:=II).\\
\end{aligned}
\]
We show that each term goes to zero.\\
The first term is  
\[
I=P_x\left(\sup_{0\leq t\leq 1}\left|  v_i(\widetilde{\dot{X}}_t)\right|>\frac{\eta}{\epsilon}\right)\] 
and observe that 
\[P_x\left(\sup_{0\leq t\leq 1}\left|  v_i(\widetilde{\dot{X}}_t)\right|>\frac{\eta}{\epsilon}\right)
\rightarrow 0  \textrm{ when } \epsilon \longrightarrow 0 \] by continuity: the map 
$t\longmapsto v_i(\widetilde{\dot{X}}_t)$ is continuous because $v_i$ is in the extended domain of $\widetilde{\dot{\xi}}$ (see theorem 2.17 of \cite{FUK}). 

The second term is equal to
\[
II=P_x\left(\sup_{\epsilon^2\leq t\leq 1}\left| \epsilon v_i(\frac{\widetilde{\dot{X}}^{(\epsilon)}_t}{\epsilon})\right|>\eta\right)
=P_x\left(\sup_{1\leq t\leq \epsilon^{-2}}\left| \epsilon v_i(\widetilde{\dot{X}}_t)\right|>\eta\right).
\]

By the Markov property, the existence and the  boundedness of the density at $t=1$,   we get that:  

\[
\begin{aligned}
II&=\int_{I_0} \widetilde{\dot{p_1}}(\dx, \dy) w(\dy)P_y\left(\sup_{0\leq t\leq\epsilon^{-2}-1}\left|\epsilon v_i(\widetilde{\dot{X}}_t)\right|>\eta\right)d\dy\\ 
& \leq c  P_w\left(\sup_{0\leq t\leq\epsilon^{-2}}\left|\epsilon v_i(\widetilde{\dot{X}}_t)\right|>\eta\right). 
\end{aligned}
\]

We use the following Lemma to show that this last term goes to zero when $\epsilon$ goes to zero.

\begin{lemma}
\label{sauveur}
For any $\eta >0$ and any $f$ in the extended domain of $\widetilde{\dot{\xi}}$,  then 
\[ 
\limsup_{\epsilon\downarrow 0}P_w(\sup_{0\leq t\leq \epsilon^{-2}}\left| \epsilon f(\widetilde{\dot{X}}_{t}\right|>\eta)\leq \frac{e^1 \sqrt{\widetilde{\dot{\xi}}(f,f)}}{\eta},
\]
 where $\widetilde{\dot{\xi}}$ is the Dirichlet form associated to  the process $\widetilde{\dot{X}}$.
\end{lemma}

We claim that Lemma \ref{sauveur} implies that, for all $\eta >0$, then 
\begin{equation}\label{eq:eq} 
P_w(\sup_{0\leq t\leq  \epsilon^{-2}}\left|\epsilon v_i(\widetilde{\dot{X}}_{t})\right|>\eta)\longrightarrow 0 \textrm{   when } \epsilon\downarrow 0. 
\end{equation}

Indeed, let  $v_s=\widetilde{\dot{P}}_s v_i$. Then   $v_s$ is also in the extended domain of $\widetilde{\dot{\xi}}$ (see lemma 1.5.4 of \cite{FUK})  and we have:

\begin{equation}\label{eqeq} \begin{aligned} 
P_w(\sup_{0\leq t\leq \epsilon^{-2} }\left|\epsilon v_i(\widetilde{\dot{X}}_{t})\right|>\eta)&\leq 
P_w(\sup_{0\leq t\leq \epsilon^{-2}}\left|\epsilon (v_i-v_s)(\widetilde{\dot{X}}_t) \right|>\frac{\eta}{2})\\&+
P_w(\sup_{0\leq t\leq \epsilon^{-2}}\left|\epsilon v_s(\widetilde{\dot{X}}_t)\right|>\frac{\eta}{2}).\end{aligned}\end{equation}

Note that $v_s(\dx)=\int_{I_0} v_i(\dy)\widetilde{\dot{p}}_s(\dx, \dy) w(\dy) d\dy \leq c(s) \left\|v_i\right\|_{L^2(I_0;  w(d\dx)) } a.e $ $ \dx\in I_{0}$ where\\ 
$c(s)= \sup_{\dx,\dy\in I_{0}} \widetilde{\dot{p_s}}(\dx, \dy)$. Therefore the second term in (\ref{eqeq}) vanishes when $\epsilon$ is small enough. 
By  Lemma \ref{sauveur} applied to the function $v_i-v_s$, 
\[ \limsup_{\epsilon\rightarrow 0} P_w(\sup_{0\leq t\leq\epsilon^{-2} }\left|\epsilon (v_i-v_s)(\widetilde{\dot{X}}_t )\right|>\frac{\eta}{2})\leq 2\frac{e^1}{\eta} \sqrt{\widetilde{\dot{\xi}} ( v_i-v_s;  v_i-v_s)}.\] 
This last bound holds for any $s>0$ and 
\[ \lim_{s\rightarrow 0} \widetilde{\dot{\xi}} ( v_i-v_s;  v_i-v_s)=0\] 
as follows from Lemma 1.5.4 of \cite{FUK}.

Thus we are done with the proof of (\ref{eq:eq}) and the proof of the convergence towards zero of (II) and (\ref{correctortendsverszerouniformement}) follows. \qed 

It nevertheless remains to prove Lemma \ref{sauveur}. 

\medskip 

{\it Proof of Lemma \ref{sauveur}}.\\ 
 \textbf{Definition}\\
We recall some material from \cite{FUK}.\\
For a nearly Borel set $A$ in $I_0$, let $\sigma_A =\inf \left\{t: \widetilde{\dot{X}}_t \in A\right\}$ and $p_A^\epsilon(\dx)=E_{\dx}(e^{-\epsilon^2\sigma_A})$.\\
Let \[\mathcal{L}_A=\left\{u\in H^1_e(I_0): u\geq 1\textrm{ q.e on A} \right\}.\] 
(q.e. means 'quasi everywhere'.) 
By theorem 2.15 (ii) of \cite{FUK},  $p_A^\epsilon(x)$ is the unique element of $\mathcal{L}_A$  minimizing $\dot{\widetilde{\xi}}_{\epsilon^2}(u,u)$ on $\mathcal{L}_A$.

We let $U_\epsilon 1 = \int_{0}^{+\infty} e^{-\epsilon^2 s} \widetilde{\dot{P}_s}1\, ds$ be the resolvent of the semigroup $(\widetilde{\dot{P}}_s)_{s> 0}$ applied to the constant function $1$. 
It  satisfies:
\begin{equation} 
\begin{aligned}
\label{schwartz}
\int_{I_0}E_{\dx}(e^{-\epsilon^2\sigma_A})w (\dx) d\dx &=\widetilde{\dot{\xi}}_{\epsilon^2}(p_A^\epsilon, U_\epsilon 1)\\
&\leq \sqrt{\widetilde{\dot{\xi}}_{\epsilon^2}( U_\epsilon 1, U_\epsilon 1)}\sqrt{\widetilde{\dot{\xi}}_{\epsilon^2}(p_A^\epsilon, p_A^\epsilon)}\\  \textrm{    by Cauchy-Schwarz inequality}.
\end{aligned}
\end{equation}

Apply this inequality to $A=\left\{ \dx\in I_0: \left|f(\dx)\right|>\eta \right\}$. 
We note that since $f\in H^1_e(I_0)$, then  $\frac{f}{\eta}\geq 1$ q.e. on  $A$\\
Thus, $ \frac{f}{\eta}\in \mathcal{L}_A $ and we obtain that 
\[ 
\widetilde{\dot{\xi}}_{\epsilon^2}(p_A^\epsilon, p_A^\epsilon)\leq {\eta^{-2}}{\widetilde{\dot{\xi}}_{\epsilon^2}(f,f)} .
 \]
Moreover, we can write:
\begin{displaymath}
 \begin{aligned}
 P_w(\sup_{0\leq t\leq \epsilon^{-2}}\left|f(\widetilde{\dot{X}}_{t}\right|>\eta)&\leq P_w(\frac{1}{\epsilon^2}\geq \sigma_A)\\
 &=\int_{I_0}P_{\dx}(\frac{1}{\epsilon^2}\geq \sigma_A) w(d\dx)\\
 &=\int_{I_0}P_{\dx}(e^{1-\epsilon^2 \sigma_A}\geq 1) w(d\dx)\\
 &\leq e^1\int_{I_0}E_{\dx}(e^{-\epsilon^2\sigma_A})w(d\dx)\\
 \end{aligned}
 \end{displaymath}
We deduce from inequality  (\ref{schwartz}) above that:
\[ P_w(\sup_{0\leq t\leq \epsilon^{-2}}\left|f(\widetilde{\dot{X}}_{t}\right|>\eta)\leq \frac{e^1}{\eta} \sqrt{\widetilde{\dot{\xi}}_{\epsilon^2}( U_\epsilon 1, U_\epsilon 1)}\sqrt{ \widetilde{\dot{\xi}}_{\epsilon^2}(f,f) }. \]
We obviously have  $\widetilde{\dot{\xi}}_{\epsilon^2}( U_\epsilon 1, U_\epsilon 1 ) = {\epsilon^{-2}}$  and therefore  
\[  P_w(\sup_{0\leq t\leq \epsilon^{-2}}\left|f (\widetilde{\dot{X}}_{t}\right|>\eta)\leq \frac{e^1}{\eta} \frac{1}{\epsilon}\sqrt{\widetilde{\dot{\xi}}_{\epsilon^2}(f,f) }.\]

Replacing  $\eta$ by $\frac{\eta}{\epsilon}$, we obtain 
\[  P_w(\sup_{0\leq t\leq \epsilon^{-2}}\left|\epsilon f (\widetilde{\dot{X}}_{t}\right|>\eta)\leq \frac{e^1}{\eta}\sqrt{\widetilde{\dot{\xi}}_{\epsilon^2}(f,f) },\] 
and Lemma \ref{sauveur} is proved letting $\epsilon$ tend to $0$.

\qed

\medskip

As in \cite{PM}, we now deduce the invariance principle for $X^{(\epsilon)}$ from the invariance principle for $\widetilde{X}^{(\epsilon)}$.

\medskip 

\textbf{B/} \textbf{Invariance principle for } $X$.

In this part of the work, we deduce from the invariance principle for $\widetilde{X}$ that  the rescaled process $\left(X^\epsilon(t)=\epsilon X (\frac{t}{\epsilon^2})\right)$ converges in distribution to a Brownian motion. We recall that a family of continuous processes $(Y^\epsilon)$ is tight under $P$ if and only if it satisfies the following compactness criterion: 
\begin{equation}
\label{compact}
\lim_{\gamma\downarrow 0}\limsup_{\epsilon \downarrow 0}P\left(\sup_{\substack{\left|t-s\right|\leq\gamma\\ 0<s,t< T}}\left|Y^\epsilon_t-Y^\epsilon_s\right|> R\right)=0
\end{equation}
for all $T>0$ and $R>0$ (see \cite{PB}, Theorem 7.5). 

Consider the two sequences of processes $({X}^{(\epsilon)})$  and $({\widetilde{X}^{(\epsilon)}})$.  
We recall that by definition of $\widetilde{X}$:
\begin{displaymath} 
\widetilde{X}^{(\epsilon)}(t)=X^{(\epsilon)}\left(\epsilon^2(A)^{-1}(t\epsilon^{-2})\right)
\end{displaymath}
Define  $A^\epsilon_t:=\epsilon^2 A_{t/\epsilon^2}$.  

The large time asymptotic of the time changed $A$ is easily deduced from the ergodic theorem as stated in the following Lemma:
\begin{lemma}
\label{ergo}
There exists a constant $k$ such that, under $P_x$  for almost any $x$, the sequence of processes $A^\epsilon$ almost surely converges to the process $(kt;t\geq 0)$ uniformly on any compact i.e., 
for all $T$,  
\begin{equation}
\label{fat}
\sup_{t\in [0,T]}\left|A^\epsilon(t)-k t\right|\longrightarrow_{\epsilon\rightarrow0} 0 
\end{equation} 
$P_x$ a.e. for almost all $x$. 
\end{lemma} 

\begin{proof} 

The ergodic theorem implies that 
\begin {displaymath}\begin{aligned}
\frac{A(t)}{t}= &\frac{\int_{0}^{t} w(\dot{X}_s)e^{V(\dot{X}_s)}ds}{t}\\ 
\longrightarrow_{t\rightarrow \infty} &(\int_{I_0} e^{-V(\dx)}d\dx)^{-1}\int_{I_0} e^{V(\dx)} w(\dx)e^{-V(\dx)}d\dx\\ 
=&(\int_{I_0} e^{-V(\dx)}d\dx)^{-1}\int_{I_0} w(\dx)d\dx:=k
\end{aligned} \end{displaymath}
$P_x$ a.e. for almost all $x$.

Observe that the map $t\longmapsto k t $ is continuous in $[0,T]$ and, 
for all $\epsilon$, the map $t\longmapsto A^\epsilon_t$ is non-decreasing.

Thus Dini's theorem applies and we deduce the uniform convergence from the pointwise convergence. 

\end{proof} 

\begin{lemma}\label{prox} 
For all $T>0$ and all $R>0$, we have 
\[\lim_{\epsilon\rightarrow 0}P_x(\sup_{t\in [0,T]}\left| X^{(\epsilon)}_t-\widetilde{X}^{(\epsilon)}_{kt}\right|>R)=0,\] 
for almost all $x$. 
\end{lemma} 

\begin{proof} 

Since $A(t)$ is bijective (continuous and strictly monotone),  we have:
\[
\widetilde{X}^{(\epsilon)}_t=X^{(\epsilon)} \left(\epsilon^2 (A)^{-1}(t\epsilon^{-2})\right)\Leftrightarrow X^{(\epsilon)}_t=\widetilde{X}^{(\epsilon)} \left(\epsilon^2 A(t\epsilon^{-2})\right).
\] 

Choose $\theta>0$. If $\sup_{t\in [0,T]}\left| X^{(\epsilon)}_t-\widetilde{X}^{(\epsilon)}_{kt}\right|>R$, then either $\sup_{t\in [0,T]}\left| A^{(\epsilon)}_t-{kt}\right|>\theta$ 
or $\sup_{t\in [0,kT];\vert t-s\vert\leq\theta}\left| \widetilde{X}^{(\epsilon)}_t-\widetilde{X}^{(\epsilon)}_{s}\right|>R$. 

Lemma \ref{ergo} implies that the probability of the first event tends to $0$ as $\epsilon$ goes to $0$. The tightness of the sequence 
$(\widetilde{X}^{(\epsilon)})$, see (\ref{compact}), 
ensures that the probability of the second event can be made as small as wanted by taking $\theta$ close to $0$. 
 
\end{proof} 

The invariance principle for the sequence $(X^{(\epsilon)})$, i.e. Theorem \ref{th}, now clearly follows from Lemma \ref{prox} and Proposition \ref{pixtilde}.

\section{Conclusion:}

We have proved a quenched invariance principle for diffusions evolving in a periodic potential, without smoothness assumptions 
and without uniform boundedness assumptions on the potential.

\begin{rem}
We note that if we consider the more general case: 
\[L= div( A\nabla)\]
where A(x) is a $d*d$-symmetric matrix satisfying the following hypothesis:  $A$ is periodic and $A\in L^1( I_0)$; there exists $V$, measurable periodic such that:
$e^{V} \in L^1(I_0)$ and \\$A\geq e^{-V} Id$, then the result of this paper holds for the diffusions associated with $L$.
\end{rem}

{\it Aknowledgement The authors  would like to thank Andrey Piatnitsky from Lebedev Physical Institute of Russian academy of Sciences and Narvik Institute of technology (Norway), and the organizers of the meeting "Heat kernel, stochastic processes and functional inequalities" in Oberwolfach May 2013.}

Aknowledgement The authors  would like to thank Andrey Piatnitsky from Lebedev Physical Institute of Russian academy of Sciences and Narvik Institute of technology (Norway), and the organizers of the meeting "Heat kernel, stochastic processes and functional inequalities" in Oberwolfach May 2013.  

\medskip 

Moustapha BA    \qquad   \qquad\and      \qquad  Pierre MATHIEU\\ 
Aix Marseille Universit\'e, CNRS, Centrale Marseille, LATP, UMR 7353.\\ 
 13453 Marseille France
\end{document}